\documentclass[12pt]{amsproc}  

\usepackage{amssymb}
\usepackage{amscd}
\usepackage{amsmath}
\usepackage{graphicx}
\usepackage{eulervm}

\newtheorem{thm}{Theorem}[section]
\newtheorem*{thm-intro}{Theorem}

\newtheorem{lemma}[thm]{Lemma}

\theoremstyle{remark}

\theoremstyle{definition}
\newtheorem{definition}[thm]{Definition}

\numberwithin{equation}{section}

\DeclareMathOperator{\Hom}{Hom}
\DeclareMathOperator{\shHom}{\underline{Hom}}
\DeclareMathOperator{\Isom}{Isom}

\DeclareMathOperator{\shAut}{\underline{Aut}}
\DeclareMathOperator{\End}{End}
\DeclareMathOperator{\shEnd}{\underline{End}}
\DeclareMathOperator{\Der}{Der}
\DeclareMathOperator{\MC}{MC}
\DeclareMathOperator{\ad}{\mathtt{ad}}
\DeclareMathOperator{\Ad}{Ad}
\DeclareMathOperator{\Conj}{Conj}
\DeclareMathOperator{\im}{Im}

\newcommand{\pr}{\mathrm{pr}}

\makeatletter
\DeclareRobustCommand\bigop[1]{%
  \mathop{\vphantom{\int}\mathpalette\bigop@{#1}}\slimits@
}
\newcommand{\bigop@}[2]{%
  \vcenter{%
    \sbox\z@{$#1\sum$}%
    \hbox{\resizebox{2mm}{\ifx#1\displaystyle1.3\fi\dimexpr\ht\z@+\dp\z@}{$\m@th#2$}}%
  }%
}
\makeatother
\DeclareMathAlphabet\mathbfcal{OMS}{cmsy}{b}{n}
\newcommand{\MI}{\DOTSB\bigop{\mathbfcal{P}}}

\begin{document}

\title{Comparison of spaces associated to DGLA via higher holonomy}

\author[P.Bressler]{Paul Bressler}
\address {Departamento de Matem\'aticas, Universidad de Los Andes, Bogot\'a, Colombia} \email{paul.bressler@gmail.com}

\author[A.Gorokhovsky]{Alexander Gorokhovsky}
\address{Department of Mathematics, UCB 395,
University of Colorado, Boulder, CO~80309-0395, USA}
\email{Alexander.Gorokhovsky@colorado.edu}

\author[R.Nest]{Ryszard Nest}
\address{Department of Mathematics,
Copenhagen University, Universitetsparken 5, 2100 Copenhagen, Denmark}
 \email{rnest@math.ku.dk}

\author[B.Tsygan]{Boris Tsygan}
\address{Department of
Mathematics, Northwestern University, Evanston, IL 60208-2730, USA}
\email{b-tsygan@northwestern.edu}

\thanks{R.~Nest was supported by the Danish National Research Foundation through the Centre for Symmetry and Deformation (DNRF92).}

\begin{abstract}
For a differential graded Lie algebra $\mathfrak{g}$ whose components vanish in degrees below $-1$ we construct an explicit equivalence between the nerve of the Deligne $2$-groupoid and the simplicial set of $\mathfrak{g}$-valued differential forms introduced by V.~Hinich. The construction uses the theory of non-abelian multiplicative integration.
\end{abstract}

\maketitle

\section{Introduction}
The principal result of the present note is a construction of a weak equivalence of two spaces (simplicial sets) naturally associated with a nilpotent differential grade Lie algebra subject to certain restrictions.

To a nilpotent DGLA $\mathfrak{g}$ which satisfies the additional condition
\begin{equation}\label{vanishing condition}
\text{$\mathfrak{g}^i = 0$ for $i<-1$}
\end{equation}
P.~Deligne \cite{Del} and, independently, E.~Getzler \cite{G1} associated a (strict) $2$-groupoid which we denote $\MC^2(\mathfrak{g})$ and refer to as the Deligne $2$-groupoid. In an earlier paper by the authors (\cite{BGNT0}) it was shown that the simplicial nerve $\mathfrak{N}\MC^2(\mathfrak{g})$ of the $2$-groupoid $\MC^2(\mathfrak{g})$, $\mathfrak{g}$ a nilpotent DGLA satisfying \eqref{vanishing condition} is weakly equivalent to another simplicial set, denoted $\Sigma(\mathfrak{g})$, introduced by V. Hinich \cite{H1} by constructing natural maps from $\mathfrak{N}\MC^2(\mathfrak{g})$ and $\Sigma(\mathfrak{g})$ to a third simplicial set and showing that they are equivalences.

In this note we outline a construction of the map
\begin{equation}\label{intro: integral}
\mathbb{I}(\mathfrak{g}) \colon \Sigma(\mathfrak{g}) \to \mathfrak{N}\MC^2(\mathfrak{g})
\end{equation}
omitting technical details. The main tool in our construction is the theory of non-abelian multiplicative integration on surfaces.

In \cite{Y}, Yekutieli constructed a theory of multiplicative integration on surfaces based on Riemann products(cf. also Aref'eva's work \cite{A}; a sketch of a construction, also based on Riemann products, is due to Kontsevich, cf. \cite{K}) and does not apply in the context at hand. In \cite{Kap} Kapranov develops a theory of higher holonomy based on Chen's iterated integrals. In the present work we take the approach based on the fundamental theorem of calculus.

The paper is organized as follows. In Section \ref{section: calculus} we recall some basic notions of differential and integral calculus. In Section \ref{section: diff geom} we review the theory of connections and holonomy. In Section \ref{section: two dim hol} we define the two-dimensional holonomy of connection-curvature pairs and discuss some of its properties. In Section \ref{section: integral} we recall the definitions of the simplicial sets $\mathfrak{N}\MC^2(\mathfrak{g})$ and $\Sigma(\mathfrak{g})$, construct the map \eqref{intro: integral} and state its properties in Theorem \ref{thm: big I} and Theorem \ref{thm: I is an equivalence}.

\subsection{Notation}
Throughout the paper
\begin{itemize}
\item $k$ is a field of characteristic zero

\item for a variety $X$ over $k$: $\mathbb{A}^n_X := X\times\mathbb{A}^n_k$; for a commutative $k$-algebra $R$: $\mathbb{A}^n_R := \mathbb{A}^n_{\operatorname{Spec}R}$

\item for a variety $\pr_X \colon Y \to X$ over $X$ and a $\mathcal{O}_X$-module $\mathcal{E}$ we denote by $\mathcal{E}_Y$ the $\mathcal{O}_Y$-module $\pr_X^*\mathcal{E}$

\item for an $X$-variety $\colon Y \to X$ over $X$ we denote by $\mathtt{d}_{/X}\colon \mathcal{O}_Y \to \Omega^1_{Y/X}$ the relative differential

\item for an nilpotent $\mathcal{O}_X$-Lie algebra $\mathfrak{g}$ we denote by $\exp(\mathfrak{g})$ the algebraic group over $X$ defined as $\exp(\mathfrak{g})(Y) = \mathfrak{g}_Y(Y)$ equipped with the group structure given by the Baker-Campbell-Hausdorff series

\item for a group $G$ and $g\in G$ we denote $\Conj(g)$ the automorphism $h \mapsto ghg^{-1}$
\end{itemize}

\section{Calculus}\label{section: calculus}
\subsection{Logarithmic derivatives}
Suppose that $\mathfrak{g}$ is a nilpotent Lie algebra. For $D\in\Der(\mathfrak{g})$, $a = \exp(x)\in\exp(\mathfrak{g})$, $x \in \mathfrak{g}$. Define $D\log(a)\in \mathfrak{g}$ by
\[
D\log(a) := (Da)\cdot a^{-1}= \left(\frac{\exp(\ad(x)) - 1}{\ad(x)}\right)(Dx).
\]
For $D\in\Der(\mathfrak{g})$, $a, b \in\exp(\mathfrak{g})$, $x\in\mathfrak{g}$,
\begin{eqnarray}
D\log(ab) & = & D\log(a) + \Ad_a(D\log(b)) \\
D\log(a^{-1}) & = & -\Ad_{a^{-1}}(D\log(a)) \\
\ad_x\log(a) & = & x - \Ad_a(x)
\end{eqnarray}

For $D_1,D_2\in\Der(\mathfrak{g})$, $a\in\exp(\mathfrak{g})$
\begin{multline}\label{dlog commutator}
D_1(D_2\log(a)) - D_2(D_1\log(a)) = \\ [D_1,D_2]\log(a) + [D_1\log(a), D_2\log(a)]
\end{multline}

\subsection{Initial value problems}
Let $\pr_X \colon \mathbb{A}^1_X \to X$ denote the projection so that $\pr_{X *}\mathcal{O}_{\mathbb{A}^1_X} = \mathcal{O}_X[s]$.

Suppose that $\mathcal{E}$ is a vector bundle on $X$. For
\begin{itemize}
\item $A\in\End_{\mathcal{O}_{\mathbb{A}^1_X}}(\mathcal{E}_{\mathbb{A}^1_X})$ nilpotent
\item a section $\sigma \colon X \to \mathbb{A}^1_X$ of the projection $\pr_X$
\item a section $e \in \Gamma(X;\mathcal{E})$
\end{itemize}
there exists a unique section $u \in \Gamma(\mathbb{A}^1_X; \mathcal{E}_{\mathbb{A}^1_X})$ which satisfies
\[
\begin{cases}
\dfrac{\partial u}{\partial s} + A(u) = 0 \\
\sigma^*(u) = e .
\end{cases}
\]

\subsection{Integration}
Let $x = s\otimes 1$ and $y = 1\otimes s$ denote the coordinate functions on $\mathbb{A}^1_X\times_X\mathbb{A}^1_X$. Let $\Delta$ denote the subvariety defined by the equation $x=y$.

For $\alpha \in \Omega^1_{\mathbb{A}^1_X/X}$ we denote by
\[
\displaystyle{\int\limits_x^y \alpha} \in \Gamma(\mathbb{A}^1_X\times_X\mathbb{A}^1_X;\mathcal{O}_{\mathbb{A}^1_X\times_X\mathbb{A}^1_X})
\]
the solution of the initial value problem
\[
\begin{cases}
\dfrac{\partial u}{\partial y} = \pr_2^*\alpha \\
u\vert_\Delta = 0
\end{cases}
\]
in the unknown $u = u(x,y) \in \Gamma(\mathbb{A}^1_X\times_X\mathbb{A}^1_X;\mathcal{O}_{\mathbb{A}^1_X\times_X\mathbb{A}^1_X})$.

\section{Differential geometry}\label{section: diff geom}

\subsection{Connections and curvature}
Suppose that $\mathfrak{g}$ is a nilpotent $\mathcal{O}_X$-Lie algebra. For $\mathrm{A} \in \Gamma(\mathbb{A}^k_X; \Omega^1_{\mathbb{A}^k_X/X}\otimes\mathfrak{g})$ 
we consider connection $\nabla := \mathtt{d}_{/X} + \mathrm{A}$. Its curvature is defined by $F = F(\nabla) := \mathtt{d}_{/X}\mathrm{A} + \dfrac12 [\mathrm{A},\mathrm{A}] \in \Gamma(\mathbb{A}^k_X; \Omega^2_{\mathbb{A}^k_X/X}\otimes\mathfrak{g})$.

Suppose now that $\mathcal{E}$ is an $\mathcal{O}_X$-module and $\rho\colon \mathfrak{g} \to \shEnd_{\mathcal{O}_X}(\mathcal{E})$ is a representation of $\mathfrak{g}$. The connection $\nabla$ gives rise to the connection $\nabla^\mathcal{E}$ on $\mathcal{E}_{\mathbb{A}^k_X}$.

The sheaf $\Omega^1_{\mathbb{A}^k_X/X}$ of relative differentials is equipped with the canonical relative connection. Hence, the connection $\nabla$ induces a connection on the $\mathcal{O}_{\mathbb{A}^k_X}$-module $\Omega^1_{\mathbb{A}^k_X/X}\otimes\mathcal{E}$ which will be denoted $\nabla^\mathcal{E}$ as well.

An example of the above situation is given by the adjoint representation $\ad \colon \mathfrak{g} \to \Der(\mathfrak{g})$ giving rise to the connection $\nabla^\mathfrak{g}$.

\subsection{Logarithmic covariant derivative}

For $u = \exp(x)\in\exp(\mathfrak{g})(\mathbb{A}^k_X)$, let
\[
d \log (u):= \left(\frac{\exp(\ad(x)) - 1}{\ad(x)}\right)(\mathtt{d}x)
\]
and
\[
\nabla\log(u) := d \log u + \mathrm{A} .
\]

\begin{lemma}
For $\xi_1, \xi_2 \in \mathcal{T}_{\mathbb{A}^k_X/X}$ such that $[\xi_1, \xi_2] = 0$ the identity
\begin{multline}\label{commutator curvature}
\nabla^\mathfrak{g}_{\xi_1}\nabla_{\xi_2} \log u-
\nabla^\mathfrak{g}_{\xi_2}\nabla_{\xi_1} \log u
= \\
[\nabla_{\xi_1} \log u, \nabla_{\xi_2} \log u]+
\iota_{\xi_1}\iota_{\xi_2}F
\end{multline}
holds.
\end{lemma}

\subsection{Holonomy}
Suppose that $\mathfrak{g}$ is a nilpotent $\mathcal{O}_X$-Lie algebra. For $\mathrm{A} \in \Gamma(\mathbb{A}^1_X; \Omega^1_{\mathbb{A}^1_X/X}\otimes\mathfrak{g})$, $\nabla := \mathtt{d}_{/X} + \mathrm{A}$ we denote by
\[
(x,y) \mapsto \MI_x^y \nabla \in \exp(\mathfrak{g})(\mathbb{A}^1_X\times_X\mathbb{A}^1_X)
\]
the solution of the initial value problem
\[
\begin{cases}
\pr_2^*\nabla_{\frac{\partial~}{\partial y}}\log(u) = 0 \\
u\vert_\Delta = 1 ,
\end{cases}
\]
where $\Delta$ is the subvariety defined by the equation $x=y$.

Suppose that $\mathcal{E}$ is an $\mathcal{O}_X$-module and $\rho\colon \mathfrak{g} \to \shEnd_{\mathcal{O}_X}(\mathcal{E})$ is a representation of $\mathfrak{g}$ which is nilpotent in the sense that there exists $N$ such that $\im(\rho)^N = 0$. The representation $\rho$ induces the map of groups $\exp(\rho) \colon \exp(\mathfrak{g}) \to \shAut(\mathcal{E})$. Since the two compositions $\mathbb{A}^1_X\times_X\mathbb{A}^1_X \xrightarrow{\pr_i} \mathbb{A}^1_X \xrightarrow{\pr_X} X$, where $\pr_i$ are the canonical projections, are equal it follows that $\shHom_{\mathcal{O}_{\mathbb{A}^1_X \times_X \mathbb{A}^1_X}}(\pr_1^*\mathcal{E}_{\mathbb{A}^1_X}, \pr_2^*\mathcal{E}_{\mathbb{A}^1_X}) = \shEnd_{\mathcal{O}_{\mathbb{A}^1_X \times_X \mathbb{A}^1_X}}(\mathcal{E}_{\mathbb{A}^1_X \times_X \mathbb{A}^1_X})$ and there is a canonical map
\begin{equation}\label{exp to isom}
\exp(\mathfrak{g})(\mathbb{A}^1_X\times_X\mathbb{A}^1_X) \to \Isom(\pr_1^*\mathcal{E}_{\mathbb{A}^1_X}, \pr_2^*\mathcal{E}_{\mathbb{A}^1_X}) .
\end{equation}

We denote by $\MI\limits_x^y \nabla^\mathcal{E}$ the image of $\MI\limits_x^y \nabla$ under \eqref{exp to isom}.

\begin{lemma}\label{identities} The following useful identities hold.
\[
\left(\MI_x^y \nabla^\mathcal{E}\right)\frac{\partial~}{\partial y}\left(\MI_x^y \nabla^\mathcal{E}\right)^{-1} = \nabla^\mathcal{E}_{\frac{\partial~}{\partial y}}
\]
\[
\left(\MI_x^y \nabla\right) \cdot\left( \MI_y^z \nabla\right) = \MI_x^z \nabla
\]
\end{lemma}

\subsection{Holonomy and curvature}
Suppose that $\mathfrak{g}$ is a nilpotent $\mathcal{O}_X$-Lie algebra, $\mathrm{A} \in \Gamma(\mathbb{A}^2_X; \Omega^1_{\mathbb{A}^2_X/X}\otimes\mathfrak{g})$, $\nabla := \mathtt{d}_{/X} + \mathrm{A}$, $F = F(\nabla) := \mathtt{d}_{/X}\mathrm{A} + \dfrac12 [\mathrm{A},\mathrm{A}] \in \Gamma(\mathbb{A}^2_X; \Omega^2_{\mathbb{A}^2_X/X}\otimes\mathfrak{g})$.

The second projection $\mathbb{A}^2_X \to \mathbb{A}^1_X$ identifies $\mathbb{A}^2_X$ with $\mathbb{A}^1_{\mathbb{A}^1_X}$. Let $\overline{\mathrm{A}} \in \Gamma(\mathbb{A}^2_X; \Omega^1_{\mathbb{A}^2_X/\mathbb{A}^1_X}\otimes\mathfrak{g})$ denote the image of $\mathrm{A}$,  $\overline{\nabla} := d_{/\mathbb{A}^1_X} + \overline{\mathrm{A}}$. Hence, the holonomy
\[
((x_1,y_2),(y_1,y_2)) \mapsto
\MI_{x_1}^{y_1} \overline{\nabla} \in \exp(\mathfrak{g})(\mathbb{A}^2_X\times_{\mathbb{A}^1_X}\mathbb{A}^2_X)
\]
is defined and will be denoted, by abuse of notation, by $\displaystyle{\MI_{x_1}^{y_1} \nabla}$.

\begin{lemma}
\begin{equation}\label{der hol param}
\nabla_{\frac{\partial~}{\partial y_2}}\log\left(\MI_{x_1}^{y_1} \nabla\right) = \int\limits_{t=x_1}^{t=y_1}(\MI_{y_1}^t \nabla^\mathfrak{g})^{-1}(\iota_{\frac{\partial~}{\partial y_2}}F)
\end{equation}
\end{lemma}
\begin{proof}
We will show that both sides of \eqref{der hol param} are solutions of the initial value problem
\[
\begin{cases}
\nabla^\mathfrak{g}_{\frac{\partial~}{\partial y_1}}\Phi =  \iota_{\frac{\partial~}{\partial y_1}} \iota_{\frac{\partial~}{\partial y_2}}F\\
\Phi\vert_{x_1=y_1} = 0 .
\end{cases}
\]
It is clear that both sides of \eqref{der hol param} satisfy the initial condition. By \eqref{commutator curvature},
\begin{multline*}
\nabla^\mathfrak{g}_{\frac{\partial~}{\partial y_1}}\nabla_{\frac{\partial~}{\partial y_2}} \log \left(\MI_{x_1}^{y_1} \nabla\right)-
\nabla^\mathfrak{g}_{\frac{\partial~}{\partial y_2}}\nabla_{\frac{\partial~}{\partial y_1}} \log \left(\MI_{x_1}^{y_1} \nabla\right)
=\\
[\nabla_{\frac{\partial~}{\partial y_1}} \log \left(\MI_{x_1}^{y_1} \nabla\right), \nabla_{\frac{\partial~}{\partial y_2}} \log \left(\MI_{x_1}^{y_1} \nabla\right)]+
\iota_{\frac{\partial~}{\partial y_1}}\iota_{\frac{\partial~}{\partial y_2}}F
\end{multline*}
 By definition, $\nabla_{\frac{\partial~}{\partial y_1}}\log\left(\MI\limits_{x_1}^{y_1} \nabla\right) = 0$. Hence,
\[
\nabla^\mathfrak{g}_{\frac{\partial~}{\partial y_1}}\nabla_{\frac{\partial~}{\partial y_2}} \log \left(\MI_{x_1}^{y_1} \nabla\right)=\iota_{\frac{\partial~}{\partial y_1}}\iota_{\frac{\partial~}{\partial y_2}}F
\]
which is to say, the left-hand side satisfies the differential equation. Differentiating the right-hand side we obtain
\begin{multline*}
\nabla^\mathfrak{g}_{\frac{\partial~}{\partial y_1}}\int\limits_{x_1}^{y_1}(\MI_{y_1}^y \nabla^\mathfrak{g})^{-1}(\iota_{\frac{\partial~}{\partial y_2}}F) \\
= \left(\MI_{x_1}^{y_1}\nabla^\mathfrak{g}\right)\frac{\partial~}{\partial y_1}\left(\MI_{x_1}^{y_1}\nabla^\mathfrak{g}\right)^{-1}\int\limits_{x_1}^{y_1}(\MI_{y_1}^y \nabla^\mathfrak{g})^{-1}(\iota_{\frac{\partial~}{\partial y_2}}F) \\
= \left(\MI_{x_1}^{y_1}\nabla^\mathfrak{g}\right)\frac{\partial~}{\partial y_1} \int\limits_{x_1}^{y_1}(\MI_{x_1}^y \nabla^\mathfrak{g})^{-1}(\iota_{\frac{\partial~}{\partial y_2}}F) = \left(\MI_{x_1}^{y_1}\nabla^\mathfrak{g}\right)\iota_{\frac{\partial~}{\partial y_1}}(\MI_{x_1}^{y_1} \nabla^\mathfrak{g})^{-1}(\iota_{\frac{\partial~}{\partial y_2}}F) \\
= \iota_{\frac{\partial~}{\partial y_1}}\iota_{\frac{\partial~}{\partial y_2}}F
\end{multline*}
Thus, the right-hand side satisfies the differential equation as well.
\end{proof}

For a nilpotent $\mathcal{O}_X$-Lie algebra $\mathfrak{g}$, $\mathrm{A} \in \Gamma(\mathbb{A}^k_X; \Omega^1_{\mathbb{A}^k_X/X}\otimes\mathfrak{g})$, $\nabla := \mathtt{d}_{/X} + \mathrm{A}$, a pair of points $(x_1,\ldots,x_k), (y_1,\ldots, y_k) \in \mathbb{A}^k_X(X)$ and an index $1\leqslant j\leqslant k$ such that $x_i = y_i$ for $i\neq j$ let
\[
\MI_{x_j}^{y_j}\nabla := \MI_{x_j}^{y_j}\overline{\nabla}^j ,
\]
where $Y = \mathbb{A}^{k-1}_X$ and $\mathbb{A}^k_X \to Y$ is the projection along the $j^{\text{th}}$ coordinate, $\overline{\nabla}^j = \mathtt{d}_{/Y} + \overline{\mathrm{A}}^j$, $\overline{\mathrm{A}}^j \in \Gamma(\mathbb{A}^1_Y; \Omega^1_{\mathbb{A}^k_Y/Y}\otimes\mathfrak{g})$ is the image of $\mathrm{A}$.

\subsection{Special broken lines}
\begin{definition}\label{def: special broken line}
A \emph{special broken line (of length $n$)} $\gamma$ in $\mathbb{A}^k_X$ is a sequence of points $(p^{(i)})_{i=0}^n$, $p^{(i)} = (p^{(i)}_1,\ldots, p^{(i)}_k)$ such that for any $0 \leqslant i < n-1$ there exists a $1\leqslant j_i\leqslant k$ such that $p^{(i)}_l = p^{(i+1)}_l$ for all $l\neq j_i$.

A special broken line is \emph{closed} if $p^{(0)} = p^{(n)}$.
\end{definition}

For a nilpotent $\mathcal{O}_{\mathbb{A}^k_X}$-Lie algebra $\mathfrak{g}$, $\mathrm{A} \in \Gamma(\mathbb{A}^k_X; \Omega^1_{\mathbb{A}^k_X/X}\otimes\mathfrak{g})$, $\nabla := \mathtt{d}_{/X} + \mathrm{A}$ and a special broken line $\gamma = (p^{(i)})_{i=0}^n$ in $\mathbb{A}^k_X$, we define the \emph {holonomy of $\nabla$ along $\gamma$} as the product
\[
\MI_\gamma \nabla := \left(\MI_{p^{(n-1)}_{j_{n-1}}}^{p^{(n)}_{j_{n-1}}} \nabla\right) \cdot \left(\MI_{p^{(n-1)}_{j_{n-2}}}^{p^{(n-1)}_{j_{n-2}}} \nabla\right)\cdot \cdots \cdot \left(\MI_{p^{(0)}_{j_{0}}}^{p^{(1)}_{j_{0}}}\nabla\right)
\]

For a special broken line $\gamma = (p^{(i)})_{i=0}^n$ we denote by $-\gamma$ the special broken line $(q^{(i)})_{i=0}^n$ with $q^{(i)} = p^{(n-i)}$.

\begin{lemma}
\[
\MI_{-\gamma} \nabla = (\MI_\gamma \nabla)^{-1}
\]
\end{lemma}

For a closed special broken line $\gamma = (p^{(i)})_{i=0}^n$ we denote by $\tau\gamma$ the closed special broken line $(q^{(i)})_{i=0}^n$ with $q^{(i)} = p^{(i+1 \pmod{n+1})}$.

\begin{lemma}
\[
\MI_{\tau\gamma} \nabla = \Conj(\MI_{p^{(0)}_{j_0}}^{p^{(1)}_{j_0}} \nabla)(\MI_\gamma \nabla)
\]
where $j_0$ is as in Definition \ref{def: special broken line}.
\end{lemma}

\section{Two-dimensional holonomy}\label{section: two dim hol}

\subsection{Connection-curvature pairs}\label{subsection: Connection-curvature pairs}
Data of a \emph{connection-curvature pair} on $\mathbb{A}^k_X$ consists, by definition, of
\begin{itemize}
\item a nilpotent crossed module\footnote{A crossed module is the same thing as a dgla concentrated in degrees 0 and -1. A \emph{nilpotent} crossed module is one which corresponds to a nilpotent dgla.} in $\mathcal{O}_X$-Lie algebras $\mathfrak{h} \xrightarrow{\delta} \mathfrak{g} \xrightarrow{\rho} \Der(\mathfrak{h})$

\item $\mathrm{A} \in \Gamma(\mathbb{A}^k_X; \Omega^1_{\mathbb{A}^k_X/X}\otimes\mathfrak{g})$

\item $\beta \in \Gamma(\mathbb{A}^k_X; \Omega^2_{\mathbb{A}^k_X/X}\otimes\mathfrak{h})$
\end{itemize}
Let $\nabla := \mathtt{d}_{/X} + \mathrm{A}$, $F = F(\nabla) := \mathtt{d}_{/X}\mathrm{A} + \dfrac12 [\mathrm{A},\mathrm{A}]$.

These data are required to satisfy
\begin{itemize}
\item $\rho(F) = \ad(\beta)$,
\item the Bianchi identity $\nabla\beta = 0$.
\end{itemize}

\subsection{Two-dimensional holonomy}\label{subsection: two dim hol}
For a connection-curvature pair $(\nabla, \beta)$ on $\mathbb{A}^2_X$ we denote by
\[
\MI_{x_2}^{y_2}\MI_{x_1}^{y_1} (\nabla,\beta) \in \exp(\mathfrak{h})(\mathbb{A}^2_X\times_X\mathbb{A}^2_X)
\]
the solution of the initial value problem
\[
\begin{cases}
\displaystyle{\nabla_{\frac{\partial~}{\partial y_2}}\log(u) = \int\limits_{x_1}^{y_1}(\MI_{y_1}^y \nabla^\mathfrak{h})^{-1}(\iota_{\frac{\partial~}{\partial y_2}}\beta)} \\
u\vert_{\Delta_1} = 1
\end{cases}
\]
where $\Delta_1$ is given by the equation $x_1 = y_1$.

For a connection-curvature pair $(\nabla, \beta)$ on $\mathbb{A}^k_X$ and
\begin{itemize}
\item points $(x_1,\ldots,x_k), (y_1,\ldots, y_k) \in \mathbb{A}^k_X(X)$;

\item indices $1\leqslant j_1, j_2 \leqslant k$ such that $x_i = y_i$ for $i\neq j_1$ and $i\neq j_2$
\end{itemize}
let
\[
\MI_{x_{j_2}}^{y_{j_2}}\MI_{x_{j_1}}^{y_{j_1}} (\nabla,\beta) := \MI_{x_{j_2}}^{y_{j_2}}\MI_{x_{j_1}}^{y_{j_1}} (\overline{\nabla}^{(j_1,j_2)},\overline{\beta}^{(j_1,j_2)})
\]
where $Y = \mathbb{A}^{k-2}_X$ and $\mathbb{A}^k_X \to Y$ is the projection along the $j_1^{\text{th}}$ and $j_2^{\text{th}}$ coordinates, $\overline{\nabla}^{(j_1,j_2)} = \mathtt{d}_{/Y} + \overline{\mathrm{A}}^{(j_1,j_2)}$, $\overline{\mathrm{A}}^{(j_1,j_2)} \in \Gamma(\mathbb{A}^2_Y; \Omega^1_{\mathbb{A}^k_Y/Y}\otimes\mathfrak{g})$ is the image of $\mathrm{A}$, $\overline{\beta}^{(j_1,j_2)} \in \Gamma(\mathbb{A}^2_Y; \Omega^1_{\mathbb{A}^k_Y/Y}\otimes\mathfrak{h})$ is the image of $\beta$.

\subsection{Rectangles}
A \emph{rectangle} $R$ in $\mathbb{A}^2_X$ is a sequence of points $(p^{(i)})_{i=0}^3$, $p^{(i)} = (p^{(i)}_1,p^{(i)}_2) \in \mathbb{A}^2_X(X)$, such that
\[
p^{(i)}_1 = p^{(i+1)}_1 \Leftrightarrow p^{(i+1)}_2 = p^{(i+2)}_2
\]
for all $i \pmod 4$. Rectangles are in one-to-one correspondence with pairs of points in $\mathbb{A}^2_X(X)$. The pair $((x_1,x_2), (y_1,y_2))$ determines the rectangle $\left((x_1,x_2),(y_1,x_2)(y_1,y_2)(x_1,y_2)\right)$.

For a rectangle $R = \left((x_1,x_2),(y_1,x_2)(y_1,y_2)(x_1,y_2)\right)$ the closed special broken line of length four
\[
\partial R := \left((x_1,x_2),(y_1,x_2)(y_1,y_2)(x_1,y_2),(x_1,x_2)\right) ,
\]
will be called the \emph{boundary of $R$}.

For a rectangle $R = \left((x_1,x_2),(y_1,x_2)(y_1,y_2)(x_1,y_2)\right)$ in $\mathbb{A}^2_X$ let
\[
\MI_R(\nabla,\beta) := \MI_{x_2}^{y_2}\MI_{x_1}^{y_1} (\nabla,\beta) .
\]

\subsection{Chains of rectangles}
A \emph{chain of rectangles} $S = (\gamma,x,y)$ in $\mathbb{A}^{k+1}_X$ is given by the following data:
\begin{itemize}
\item A special broken line $\gamma = (p^{(i)})_{i=0}^n$, $p^{(i)} = (p^{(i)}_1,\ldots, p^{(i)}_k)$
in $\mathbb{A}^k_{\mathbb{A}^1_X}$.
\item  Points $x,y \in \mathbb{A}^1_X(X)$. 
\end{itemize}
The chain of rectangles $S$ consists of rectangles  
\[R_i:= \left( (x, p^{(i-1)}_{j_{i-1}}), (y, p^{(i-1)}_{j_{i-1}}), (y, p^{(i)}_{j_{i-1}}), (x, p^{(i)}_{j_{i-1}})\right),
\]
$1\leqslant i \leqslant n$, in the notations of Definition \ref{def: special broken line}.
The
\emph{boundary} of $S$, denoted $\partial S$ is the closed special broken line
\[
\partial S := ((p^{(0)},x),(p^{(0)},y),(p^{(1)},y),\ldots,(p^{(n)},y),(p^{(n)},x),\ldots,(p^{(0)},x))
\]
in $\mathbb{A}^{k+1}_X$.

For $1\leqslant i \leqslant n$ let $\gamma_{\leqslant i} = (p^{(l)})_{l=0}^i$, $S_{\leqslant i} = (\gamma_{\leqslant i},x,y)$. Note that $\gamma_{\leqslant n} = \gamma$ and $S_{\leqslant n} = S$. 
Define
\[
\MI_{S_{\leqslant 1}}(\nabla,\beta) := \MI_{R_1} (\nabla,\beta)
\]
and for $2\leqslant i \leqslant n$ let
\[
\MI_{S_{\leqslant i}}(\nabla,\beta) := \Conj\left(\MI_{\gamma_{\leqslant i-1}}\nabla\right)\left(\MI_{R_i} (\nabla,\beta)\right)\cdot\MI_{S_{\leqslant i-1}}(\nabla,\beta)
\]

\subsection{Green's theorem for chains of rectangles}

\begin{thm}[Green's Theorem]
Suppose that $S$ is a chain of rectangles in $\mathbb{A}^{k+1}_X$ and $(\nabla, \beta)$ is a connection-curvature pair as in \ref{subsection: Connection-curvature pairs}. Then,
\[
\MI_{\partial S}\nabla = \exp(\delta)(\MI_S(\nabla,\beta)) .
\]
\end{thm}

\subsection{Parallelepipeds}
A parallelepiped is a pair of points in $\mathbb{A}^3_X$.

In what follows we use that first coordinate projection to identify $\mathbb{A}^3_X$ with $\mathbb{A}^2_{\mathbb{A}^1_X}$. A parallelepiped $Q = ((x_1,x_2,x_3),(y_1,y_2,y_3))\in \mathbb{A}^3_X\times_X\mathbb{A}^3_X(X)$ gives rise to the rectangle $R = ((x_2,x_3),(y_2,x_3)(y_2,y_3)(x_2,y_3))$ in $\mathbb{A}^2_{\mathbb{A}^1_X}$ and the chain of rectangles $S = (\partial R, x_1,y_1)$.

\begin{thm}[Gauss-Ostrogradsky Theorem]\label{thm: gauss on brick}
Suppose that $Q$ is a parallelepiped in $\mathbb{A}^3_X$ and $(\nabla, \beta)$ is a connection-curvature pair as in \ref{subsection: Connection-curvature pairs}.
\[
\MI_S(\nabla,\beta) = \Conj(\left.\MI_{x_1}^{y_1}\nabla\right\vert_{(x_2,x_3)})\left(\left.\MI_R(\nabla,\beta)\right\vert_{y_1}\right)\cdot \left(\left.\MI_R(\nabla,\beta)\right\vert_{x_1}\right)^{-1}
\]
\end{thm}

\subsection{Parametrization of simplexes}
Recall that the $n$-simplex $\Delta^n$ is defined as the hyperplane in $\mathbb{A}^{n+1}$ given by the equation $t_0+\dots +t_n = 1$.

We parameterize $\Delta^1$, $\Delta^2$ and $\Delta^3$, respectively, by
\begin{enumerate}
\item[$\Delta^1$:] $t_0 = 1-x$, $t_1 = x$;

\item[$\Delta^2$:] $t_0 = 1-x_1$, $t_1 = x_1(1-x_2)$, $t_2 = x_1x_2$;

\item[$\Delta^3$:] $t_0 = 1-x_1$,$t_1 = x_1(1-x_2)$, $t_2 = x_1x_2(1-x_3)$, $t_3 = x_1x_2x_3$ .
\end{enumerate}
Let $\Phi^i$ denote the parameterization of $\Delta^i$, $i=1,2,3$.

\[
\MI_{\Delta^1}\nabla := \MI_0^1\Phi^{1*}\nabla
\]

\[
\MI_{\Delta^2}(\nabla,\beta) := \MI_0^1\MI_0^1(\Phi^{2*}\nabla,\Phi^{2*}\beta)
\]

With these notations, for a connection-curvature pair $(\nabla,\beta)$ on $\mathbb{A}^3_X$  Gauss-Ostrogradsky Theorem \ref{thm: gauss on brick} says
\begin{multline}\label{gauss for simplex}
\Conj(\MI_{\Delta^1} f^*\nabla)\left(\MI_{\Delta^2}\partial_0^*(\nabla,\beta) \right) = \\
\left(\MI_{\Delta^2}\partial_2^*(\nabla,\beta)\right)^{-1}\cdot\left(\MI_{\Delta^2}\partial_1^*(\nabla,\beta)\right)\cdot\left(\MI_{\Delta^2}\partial_3^*(\nabla,\beta)\right),
\end{multline}
where $f = \partial_3\circ\partial_2 = \partial_2\circ\partial_2$ is the map $\Delta^1 \to \Delta^3$ induced by $0 \mapsto 0$, $1 \mapsto 1$.

\section{From $\Sigma$ to $\mathfrak{N}\MC^2$}\label{section: integral}

Throughout this section we assume that $\mathfrak{g}$ is a nilpotent DGLA which satisfies $\mathfrak{g}^i = 0$ for $i \leqslant -2$. Recall that in \cite{BGNT0} we showed that the simplicial nerve $\mathfrak{N}\MC^2(\mathfrak{g})$ of the $2$-groupoid $\MC^2(\mathfrak{g})$ is weakly equivalent to another simplicial set $\Sigma(\mathfrak{g})$.
In this section we apply the theory of multiplicative integration outlined above to construct the map
\[
\mathbb{I}(\mathfrak{g}) \colon \Sigma(\mathfrak{g}) \to \mathfrak{N}\MC^2(\mathfrak{g})
\]
omitting technical details. We will rely heavily on results and notations from \cite{BGNT0}.

In what follows we denote by $\Omega_n$, $n = 0, 1, 2, \ldots$ the commutative differential graded algebra over $\mathbb{Q}$ with generators $t_0,\ldots,t_n$ of degree zero and $dt_0, \ldots, dt_n$ of degree one subject to the relations $t_0 + \cdots + t_n = 1$ and $dt_0 + \cdots + dt_n = 0$. The differential $d \colon \Omega_n \to \Omega_n[1]$ is defined by $t_i \mapsto dt_i$ and $dt_i \mapsto 0$. The assignment $[n] \mapsto \Omega_n$ extends in a natural way to a simplicial commutative differential graded algebra.

Recall that the set of Maurer-Cartan elements of $\mathfrak{g}$, denoted $\MC(\mathfrak{g})$ is defined as
\[
\MC(\mathfrak{g}) := \left\{\gamma\in\mathfrak{g}^1 \mid \delta\gamma + \frac12[\gamma,\gamma] = 0\right\} .
\]

\subsection{The simplicial set $\Sigma(\mathfrak{g})$}
For a nilpotent $L_\infty$-algebra $\mathfrak{g}$ and a non-negative integer $n$ let
\[
\Sigma_n(\mathfrak{g}) =\MC(\mathfrak{g}\otimes\Omega_n) .
\]
Equipped with structure maps induced by those of $\Omega_\bullet$ the assignment $n \mapsto \Sigma_n(\mathfrak{g})$ defines a simplicial set denoted $\Sigma(\mathfrak{g})$.

The simplicial set $\Sigma(\mathfrak{g})$ was introduced by V.~Hinich in \cite{H1} for DGLA
and used by E.~Getzler in \cite{G1} (where it is denoted $\MC_\bullet(\mathfrak{g})$) for
general nilpotent $L_\infty$-algebras.

\subsection{The Deligne 2-groupoid}\label{ss: Deligne}
We denote by $\MC^2(\mathfrak{g})$ the Deligne $2$-groupoid as
defined by P.~Deligne \cite{Del} and independently by E.~Getzler, \cite{G1}. Below we review the construction of Deligne $2$-groupoid of a nilpotent DGLA following \cite{G1, G2} and references therein.

Since $\mathfrak{g}$ is a nilpotent DGLA it follows that $\mathfrak{g}^0$ is a nilpotent Lie algebra. The unipotent group $\exp \mathfrak{g}^0$ acts on the space $\mathfrak{g}^1$ by affine transformations. The action of $\exp X$, $X\in\mathfrak{g}^0$, on $\gamma\in\mathfrak{g}^1$ is given by the formula
\begin{equation}\label{gauge transformation}
(\exp X) \cdot \gamma= \gamma- \sum_{i=0}^{\infty} \frac{(\ad
X)^i}{(i+1)!}(\delta X + [\gamma, X]) .
\end{equation}
The action of gauge transformations \eqref{gauge transformation} preserves the subset of Maurer-Cartan elements $\MC(\mathfrak{g})\subset\mathfrak{g}^1$.

We denote by $\MC^1(\mathfrak{g})$ the Deligne groupoid defined as the groupoid associated with the action of the group $\exp \mathfrak{g}^0$ by gauge
transformations on the set $\MC(\mathfrak{g})$.
%

We denote by $\MC^2(\mathfrak{g})$ the Deligne $2$-groupoid whose objects and the $1$-morphisms of $\MC^2(\mathfrak{g})$ are those of $\MC^1(\mathfrak{g})$. The 2-morphisms of $\MC^2(\mathfrak{h})$are defined as follows.

For $\gamma \in \MC(\mathfrak{g})$ let $[\cdot, \cdot]_{\gamma}$ denote the Lie bracket on $\mathfrak{g}^{-1}$ defined by
\begin{equation}\label{eq:mu-bracket}
[a,\,b]_{\gamma}=[a,\, \delta b+[\gamma, \,b]].
\end{equation}
Equipped with this bracket, $\mathfrak{g}^{-1}$ becomes a nilpotent Lie
algebra. We denote by $\exp_{\gamma} \mathfrak{g}^{-1}$ the
corresponding unipotent group, and by
\[
\exp_{\gamma} \colon \mathfrak{g}^{-1} \to \exp_{\gamma}\mathfrak{g}^{-1}
\]
the corresponding exponential map. If $\gamma_1$, $\gamma_2$ are two Maurer-Cartan
elements, then the group $\exp_{\gamma_2} \mathfrak{g}^{-1}$ acts on
$\Hom_{\MC^1(\mathfrak{g})}(\gamma_1, \gamma_2)$. For $\exp_{\gamma_2} t
\in \exp_{\gamma_2} \mathfrak{g}^{-1}$ and $\Hom_{\MC^1(\mathfrak{g})}(\gamma_1, \gamma_2)$ the action is given by
\begin{equation*}
(\exp_{\gamma_2} t) \cdot (\exp X) =
\exp(\delta t+[\gamma_2,t])\, \exp X \in \exp
\mathfrak{g}^0 .
\end{equation*}
By definition, $\Hom_{\MC^2(\mathfrak{g})}(\gamma_1, \gamma_2)$ is the groupoid associated with the above action.

\subsection{The simplicial nerve of the Deligne 2-groupoid}
We will need the following explicit description of the simplicial nerve of $\MC^2(\mathfrak{g})$. See \cite{BGNT0} for details.

\begin{lemma}\label{lemma:the two gothic Ns}
The simplicial nerve $\mathfrak{N}\MC^2(\mathfrak{g})$ admits the following explicit
description:
\begin{enumerate}
\item $\mathfrak{N}_0\MC^2(\mathfrak{g}) = \MC(\mathfrak{g})$

\item For $n\geq 1$ there is a canonical bijection between $\mathfrak{N}_n\MC^2(\mathfrak{g})$ and the set of data of the form
\[
((\mu_i)_{0\leq i\leq n}, (g_{ij})_{0\leq i < j\leq n}, (c_{ijk})_{0\leq i < j<k\leq n}) ,
\]
where
\begin{itemize}
\item $(\mu_i)$ is an $(n+1)$-tuple of objects of Maurer-Cartan elements of $\mathfrak{g}$,
\item $(g_{ij})$ is a collection of $1$-morphisms (gauge transformations) $g_{ij}\colon\mu_j\to \mu_i$
\item $(c_{ijk})$ is a collection of $2$-morphisms $c_{ijk}:g_{ij}g_{jk}\to g_{ik}$ which satisfies
\begin{equation} \label{eq:cece}
c_{ijl}c_{jkl}=c_{ikl}c_{ijk}
\end{equation}
(in the set of $2$-morphisms $g_{ij}g_{jk}g_{kl}\to g_{il}$).
\end{itemize}

\end{enumerate}

For a morphism $f\colon [m] \to [n]$ in $\Delta$ the induced structure map
$f^*\colon\mathfrak{N}_n\MC^2(\mathfrak{g}) \to \mathfrak{N}_m\MC^2(\mathfrak{g})$ is given (under the above bijection) by $f^*((\mu_i),(g_{ij}),(c_{ijk})) = ((\nu_i),(h_{ij}),(d_{ijk}))$,
where $\nu_i = \mu_{f(i)}$, $h_{ij} = g_{f(i),f(j)}$, $d_{ijk} = c_{f(i),f(j),f(k)}$.
\end{lemma}

\subsection{From $\Sigma$ to $\mathfrak{N}\MC^2$}
Recall that, for $n = 0,1,2,\ldots$, $\Sigma_n(\mathfrak{g}) = \MC(\Omega_n\otimes\mathfrak{g}) \subset \left(\Omega_n\otimes\mathfrak{g}\right)^1 = \Omega_n^0\otimes\mathfrak{g}^1 \oplus \Omega_n^1\otimes\mathfrak{g}^0 \oplus \Omega_n^2\otimes\mathfrak{g}^{-1}$. Thus, every element $\mu\in\Sigma_n(\mathfrak{g})$ is a triple: $\mu = (\mu^{0,1},\mu^{1,0},\mu^{2,-1})$.

For $n = 0,1,2,\ldots$ the map
\begin{equation}
\mathbb{I}_n(\mathfrak{g}) \colon \Sigma_n(\mathfrak{g}) \to \mathfrak{N}_n\MC^2(\mathfrak{g})
\end{equation}
is defined as follows:
\begin{enumerate}
\item[$n=0$] The map $\mathbb{I}_0(\mathfrak{g}) \colon \Sigma_0(\mathfrak{g}) = \MC(\mathfrak{g}) \to\MC(\mathfrak{g}) = \mathfrak{N}_0\MC^2(\mathfrak{g})$ is defined to be the identity.

\item[$n=1$] For $\mu = (\mu^{0,1},\mu^{1,0}) \in \Sigma_1(\mathfrak{g})$,
\[
\mathbb{I}_1(\mathfrak{g})(\mu) = ((\partial_i^*\mu^{0,1})_{i=0,}, \MI\limits_{\Delta^1}(d+\mu^{1,0})).
\]

\item[$n\geq 2$] For $\mu = (\mu^{0,1},\mu^{1,0},\mu^{2,-1}) \in \Sigma_n(\mathfrak{g})$,
\[
\mathbb{I}_n(\mathfrak{g})(\mu) = ((\mu^{0,1}_i)_{0\leqslant i \leqslant n}, (\MI\limits_{\Delta^1}(d+\mu^{1,0}_{ij})_{0\leqslant i<j \leqslant n}, (\MI\limits_{\Delta^2}(d+\mu^{1,0}_{ijk},\mu^{2,-1}_{ijk}))_{0\leqslant i<j<k \leqslant n}),
\]
where
\begin{itemize}
\item $\mu^{0,1}_i = f^*\mu^{0,1}$, where $f \colon \Delta^0 \to \Delta^n$ is the map induced by $[0] \to [n] \colon 0 \mapsto i$;

\item $\mu^{1,0}_{ij} = f^*\mu^{1,0}$, where $f \colon \Delta^1 \to \Delta^n$ is the map induced by $[1] \to [n] \colon 0 \mapsto i, 1 \mapsto j$;

\item $\mu^{2,-1}_{ijk} = f^*\mu^{2,-1}$, where $f \colon \Delta^2 \to \Delta^n$ is the map induced by $[1] \to [n] \colon 0 \mapsto i, 1 \mapsto j, 2 \mapsto k$.
\end{itemize}
\end{enumerate}

\begin{thm}\label{thm: big I}
Suppose that $\mathfrak{g}$ is a nilpotent DGLA which satisfies $\mathfrak{g}^i = 0$ for $i \leqslant -2$. The collection of maps $\mathbb{I}_n(\mathfrak{g})$, $n = 0,1,2,\ldots$, defines a morphism of simplicial sets
\begin{equation}\label{higher holonomy}
\mathbb{I}(\mathfrak{g}) \colon \Sigma(\mathfrak{g}) \to \mathfrak{N}\MC^2(\mathfrak{g})
\end{equation}
natural in $\mathfrak{g}$ which satisfies
\begin{enumerate}
\item $\mathbb{I}_0(\mathfrak{g}) \colon \Sigma_0(\mathfrak{g}) \to \mathfrak{N}_0\MC^2(\mathfrak{g})$ is the identity map after the identification $\Sigma_0(\mathfrak{g}) = \MC(\mathfrak{g}) = \mathfrak{N}_0\MC^2(\mathfrak{g})$.

\item The restriction of $\mathbb{I}$ to the subcategory of abelian algebras (a.k.a. complexes) coincides with the integration map
\begin{equation}\label{abelian integration}
\int\colon \Sigma(\mathfrak{a}) \to K(\mathfrak{a}[1]) = \mathfrak{N}\MC^2(\mathfrak{a}),
\end{equation}
$\mathfrak{a}$ abelian. Note that the integration map is a morphism of simplicial groups.

\item If $\mathfrak{a} \hookrightarrow \mathfrak{g}$ is a central subalgebra the diagram
\[
\begin{CD}
\Sigma(\mathfrak{g}) @>\mathbb{I}>> \mathfrak{N}\MC^2(\mathfrak{g}) \\
@VVV @VVV \\
\Sigma(\mathfrak{g}/\mathfrak{a})_0 @>\mathbb{I}>> \mathfrak{N}\MC^2(\mathfrak{g}/\mathfrak{a})_0
\end{CD}
\]
is a morphism of principal fibrations relative to the morphism of groups \eqref{abelian integration}, where $\Sigma(\mathfrak{g}/\mathfrak{a})_0$ (respectively, $\mathfrak{N}\MC^2(\mathfrak{g}/\mathfrak{a})_0$) denotes the image of the map $\Sigma(\mathfrak{g}) \to \Sigma(\mathfrak{g}/\mathfrak{a})$ (respectively, $\mathfrak{N}\MC^2(\mathfrak{g}) \to \mathfrak{N}\MC^2(\mathfrak{g}/\mathfrak{a})$).
\end{enumerate}
\end{thm}

Theorem \ref{thm: big I} and induction on the nilpotency length imply the following statement.

\begin{thm}\label{thm: I is an equivalence}
The map $\mathbb{I}(\mathfrak{g})$ is an equivalence.
\end{thm}

\bibliographystyle{plain}
\bibliography{bibliography}
\end{document}